\documentclass{aptpub}
\usepackage{graphicx}
\usepackage{color}

\authornames{A. Genadot and M. Thieullen}
\shorttitle{Averaging for a PDMP in Infinite Dimensions}

\newcommand{\R}{\mathbb{R}}
\newcommand{\PP}{\mathbb{P}}
\newcommand{\EE}{\mathbb{E}}
\newcommand{\Ba}{\mathcal{B}}
\newcommand{\e}{\varepsilon}
\newcommand{\D}{\Delta}
\newcommand{\ZZ}{\mathbb{N}\cap{N\mathring{I}}}
\newcommand{\RR}{\mathcal{R}}

\begin{document}

\title{Averaging for a Fully-Coupled Piecewise Deterministic Markov Process in Infinite Dimensions} 

\authorone[Laboratoire de Probabilit\'es et Mod\`eles Al\'eatoires, Universit\'e Pierre et Marie Curie, Paris 6, UMR7599.]{Alexandre Genadot}
\emailone{algenadot@gmail.com.}
\authortwo[Laboratoire de Probabilit\'es et Mod\`eles Al\'eatoires, Universit\'e Pierre et Marie Curie, Paris 6, UMR7599 .]{Mich\`ele Thieullen}
\emailtwo{michele.thieullen@upmc.fr}
\address{ Universit\'e Pierre et Marie Curie, Paris 6, Case courrier 188, 4 Place Jussieu, 75252 Paris Cedex 05, France.\\
This work has been supported by the Agence Nationale de la Recherche through the project MANDy, Mathematical Analysis of Neuronal Dynamics, ANR-09-BLAN-0008-01.}

\begin{abstract}
In this paper, we consider the generalized Hodgkin-Huxley model introduced by Austin in \cite{Austin}. This model describes the propagation of an action potential along the axon of a neuron at the scale of ion channels. Mathematically, this model is a fully-coupled Piecewise Deterministic Markov Process (PDMP) in infinite dimensions. We introduce two time scales in this model in considering that some ion channels open and close at faster jump rates than others. We perform a slow-fast analysis of this model and prove that asymptotically this 'two-time-scales' model reduces to the so called averaged model which is still a PDMP in infinite dimensions for which we provide effective evolution equations and jump rates.
\end{abstract}



\section{Introduction}

The Hodgkin-Huxley model is one of the most studied models in neuroscience since its creation in the early 50's by Hodgkin and Huxley \cite{HH}. This deterministic model was first created in order to describe the propagation of an action potential or nerve impulse along the axon of a neuron of the giant squid. In our paper, we will consider a mathematical generalization of the Hodgkin-Huxley model investigated numerically in \cite{FWL} and analytically in \cite{Austin}. This model, called the spatial stochastic Hodgkin-Huxley model in the sequel, is made of a partial differential equation (PDE) describing the evolution of the variation of the potential across the membrane, coupled with a continuous time Markov chain which describes the dynamics of ion channels which are present all along the axon. It is a Piecewise Deterministic Markov Process (PDMP) in infinite dimensions (see \cite{Buck_Ried,Wainrib2} for PDMP in infinite dimensions and also \cite{Costa1, Costa2,Crudu,Malrieu} and references therein for PDMP in finite dimension). Moreover it is fully-coupled in the sense that the evolution of the potential depends on the kinetics of ion channels and vice-versa. The role of ion channels is fundamental because they allow and amplify the propagation of an action potential. We introduce different time scales in this model considering that some ion channels open and close at faster rates than others. We perform a slow-fast analysis of this model and prove that asymptotically it reduces to the so called averaged model which is still a PDMP in infinite dimensions for which we provide effective evolution equations and jump rates. We thus reduce the complexity of the original model by simplifying the kinetics of ion channels. To the best of our knowledge no averaging results are available for PDMP in infinite dimensions.

Even if the Hodgkin-Huxley model was first introduced in neurophysiology, its paradigm is commonly used to describe the behavior of cardiac cells, see for instance \cite{Fl_al,Gr_al}. It was at first deterministic. In order to include more variability, better describe the axon itself and a certain class of phenomena, stochastic versions of this model have been proposed since the 90's, see for instance \cite{Ch_Wh,Def_Is,FWL,Fox_1}, introducing in the model two kinds of randomness. One is to include in the Hodgkin-Huxley equations an external noise modeling in this way the multitude of impulses received by one single neuron through its connections to the extraordinary complex system which is the brain. The other is to introduce in the Hodgkin-Huxley equations the intrinsic noise of the neuron due to the presence of ion channels which are intrinsically stochastic entities. The spatial stochastic Hodgkin-Huxley model which deals with the intrinsic noise of neurons pertains to the latter class.

The generalized Hodgkin-Huxley equations of \cite{Austin} belong also to a larger class of stochastic systems describing the evolution of a macroscopic variable, here the membrane potential, coupled with a Markovian dynamic modeling the kinetics of smaller entities, here ion channels. Between two successive jumps of the Markovian kinetics, the macroscopic variable follows a deterministic evolution. These models are known mathematically under the name of Piecewise Deterministic Markov Processes (PDMP) and are often called stochastic hybrid systems in the applications. The theory of PDMP in finite dimension (dealing with ordinary differential equations coupled with Markovian kinetics) has been developed in \cite{Davis_2,Davis_1}. The extension to infinite dimensions (dealing with PDEs coupled with Markovian kinetics) has been performed in \cite{Buck_Ried}. The spatial stochastic Hodgkin-Huxley model belongs to the class of PDMP in infinite dimensions since the equation describing the evolution of the membrane potential is a PDE.

Reducing the complexity of a model is of first importance since it allows to perform more tractable mathematics and numerical simulations and in this way, to go further in the understanding of the model. The reduction can be achieved by two different means. The first is to consider that the number or the population of ion channels is large enough to apply a law of large numbers argument so that we can replace the empirical proportion of ion channels in a given state by the averaged theoretical proportion of ion channels in this state. This is achieved by letting the number of ion channels go to infinity, the corresponding law of large numbers is proved in \cite{Austin} and in the context of PDMP in \cite{RTW}. The second way to reduce the complexity of the model, the one we will be interested in in this paper, is to consider that the model has intrinsically two time scales. It is a quite natural approach in this context, see for instance \cite{Fitz,Hille}.

In the present paper the number of ion channels is fixed and we consider that some ion channels have a faster dynamic than others. We perform then a slow-fast mathematical analysis of this 'two-time-scales' system. We obtain a reduced system in which the intrinsic variability of ion channels is still taken into account but with a simplified kinetic. Unlike the main result of \cite{Austin}, after reduction, we remain here at a stochastic level. The methods used in this paper are known mathematically under the name of stochastic averaging.  As far as we know, these methods have never been applied to a PDMP in infinite dimensions which we are dealing with in the present paper. For PDMPs in finite dimension, a theory of averaging has been developed in \cite{Fa_Cri} and central limit theorem has been proved in \cite{Wainrib2}.

In the infinite dimensional framework which we consider in this paper we have been able to reduce the complexity of a fully-coupled PDMP in infinite dimensions by using averaging methods. The natural step after this averaging result is to prove the central limit theorem associated to it. We wish to investigate this in a future work.

We conclude this introduction with the plan of the paper. In Section \ref{intro_math} we introduce the spatial stochastic Hodgkin-Huxley model and the formalism of PDMP in infinite dimensions. Then we introduce the two time scales in the model and prove a crucial result for the sequel. We finish this section by presenting the main assumptions on the model and our main result. Section \ref{section_proof} is devoted to the proof of our main result. In Section \ref{section_ex} we apply our result to a Hodgkin-Huxley type model as an example. At the end of the paper, an appendix provides an important lemma proved in \cite{Austin}, two results about tightness and the parameter values used in the simulation presented in Section \ref{section_ex}.

\subsection{Basic notations}

We set $I=[0,1]$. Let us define:
\[
\|f\|_H=\sqrt{\int_I (f(x))^2+(f'(x))^2dx}\quad\|f\|_{L^2(I)}=\sqrt{\int_I (f(x))^2dx}
\]
In this paper we will work with the following triplet of Banach spaces $H^1_0(I)\subset L^2(I)\subset H^{-1}(I)$. $H^1_0(I)$ will be denoted by $H$, it denotes the completion of the set of $\mathcal{C}^\infty$ functions with compact support on $I$ with respect to the norm $\|\cdot\|_H$. $L^2(I)$ endowed with the norm $\|\cdot\|_{L^2(I)}$ is the usual space of measurable and square integrable functions with respect to the Lebesgue measure on $I$. $H^{-1}(I)$ is the dual space of $H$ and will be denoted by $H^*$.

We recall here a few basic results about these three spaces.  $L^2(I)$ and $H$ are two separable Hilbert spaces endowed with their usual scalar products denoted respectively by $(\cdot,\cdot)_{L^2(I)}$ and $(\cdot,\cdot)$.  We denote by $<\cdot,\cdot>$ the duality pairing between $H$ and $H^*$.

For an integer $k\geq1$ we define the following functions on $I$:
\[
e_k(\cdot)=\frac{\sqrt2}{\sqrt{1+(k\pi)^2}}\sin(k\pi \cdot),\quad f_k(\cdot)=\sqrt2\sin(k\pi \cdot)
\]
The family $\{f_k,k\geq1\}$ (resp. $\{e_k,k\geq1\}$) is a Hilbert basis of $L^2(I)$ (resp. $H$). In $L^2(I)$ (resp. $H$), the Laplacian with zero Dirichlet boundary conditions has the following spectral decomposition:
\begin{equation*}
\D u=-\sum_{k\geq1}(k\pi)^2(u,f_k)_{L^2(I)}f_k\quad \left(\text{ resp. } \D u=-\sum_{k\geq1}(k\pi)^2(u,e_k)e_k\right) 
\end{equation*}
for $u$ in the domain $\mathcal{D}_{L^2(I)}(\D)=\{u\in L^2(I);\sum_{k\geq1}k^4 (u,f_k)^2_{L^2(I)}<\infty\}$ (resp. $\mathcal{D}(\D)=\{u\in H;\sum_{k\geq1}k^4 (u,e_k)^2<\infty\}$). We refer the reader to chapter 1, section 1.3 of \cite{Henry} for more details.

We denote by $C_P$ the (Poincar\'e) constant such that, for all $u\in H$ we have:
\[
\sup_{I}|u|\leq C_P\|u\|_H
\]

The embeddings $H\subset L^2(I)\subset H^*$ are continuous and dense. Moreover, for any $h\in L^2(I)$ and any $u\in H$: $ <h,u>\equiv(h,u)_{L^2(I)}$.

We say that a function $f:H\mapsto\R$ has a Fr\'echet derivative in $u\in H$ if there exists a bounded linear operator $T_u:H\mapsto \R$ such that:
\[
\lim_{h\to 0}\frac{f(u+h)-f(u)-T_u(h)}{\|h\|_H}=0
\]
We then write $\frac{df}{du}[u]$ for the operator $T_u$. For example, the square of the $\|\cdot\|_H$ norm and the Dirac distribution in $x\in I$ are Fr\'echet differentiable on $H$. For all $u\in H$:
\[
\frac{d\|\cdot\|^2_H}{du}[u](h)=2(u,h),\quad \frac{d\delta_x}{du}[u](h)=h(x)
\]
for all $h\in H$. Fr\'echet differentiation is stable by sum and product.

\section{Statement of the model and results}\label{intro_math}

\subsection{The spatial stochastic Hodgkin-Huxley model}

We introduce here the stochastic Hodgkin-Huxley equations considered in \cite{Austin}. Basically this spatial stochastic Hodgkin-Huxley model describes the propagation of an action potential along an axon. The axon is the part of a neuron which transmits the information received from the soma to another neuron on long distances: the length of the axon is big relative to the size of the soma. Along the axon are the ion channels which amplify and allow the propagation of the received impulse. For mathematical convenience, we assume that the axon is a segment of length one and we denote it by $I=[0,1]$. The ion channels are assumed to be in number $N\geq1$ and regularly placed along the axon at the place $\frac iN$ for $i\in\ZZ$. This distribution of the channels is certainly unrealistic but we assume it to fix the ideas and in accordance with \cite{Austin}. However the mathematics are the same if we consider any finite subset of $I$ instead of $\frac1N(\ZZ)$. Each ion channel can be in a state $\xi\in E$ where $E$ is a finite states space. For example, for the Hodgkin-Huxley model, a state can be : "receptive to sodium ions and open" (see \cite{Hille}).

The ion channels switch between states according to a continuous time Markov chain whose jump intensities depend on the local potential of the membrane, that is why the model is said to be fully-coupled. For any states $\xi,\zeta\in E$ we define by $\alpha_{\xi,\zeta}$ the jump intensity (or rate function) from the state $\xi$ to the state $\zeta$. It is a real valued function of a real variable and supposed to be, as is its derivative, Lipschitz-continuous. We assume moreover that: $0\leq\alpha_{\xi,\zeta}\leq\alpha^+$  for any $\xi,\zeta\in E$ and either $\alpha_{\xi,\zeta}$ is constant equal to zero or is strictly positive bounded below by a strictly positive constant $\alpha_-$. That is, the non-zero rate functions are bounded below and above by strictly positive constants. For a given channel, the rate function describes the rate at which it switches from one state to another. 

A possible configuration of all the $N$ ion channels is denoted by an element $r=(r(i),i\in\ZZ)$ of $\RR=E^{\ZZ}$: $r(i)$ is the state of the channel which is at the position $\frac iN$ for $i\in\ZZ$. The channels, or stochastic processes $r(i)$, are supposed to evolve independently over infinitesimal time-scales. Denoting by $u_t(\frac{i}{N})$ the local potential at point $\frac iN$ at time $t$, we have:
\begin{equation}\label{saut}
\PP(r_{t+h}(i)=\zeta|r_t(i)=\xi)=\alpha_{\xi,\zeta}\left(u_t\left(\frac{i}{N}\right)\right)h+o(h)
\end{equation}

For any $\xi\in E$ we also define the maximal conductance $c_\xi$ and the steady state potentials $v_\xi$ of a channel in state $\xi$ which are both  constants, the first being positive.

The membrane potential $u_t(x)$ along the axon evolves according to the following PDE:
\begin{equation}\label{dyn}
\partial_t u_t=\Delta u_t +\frac{1}{N}\sum_{\xi\in E}\sum_{i\in \ZZ}c_\xi 1_{\xi}(r_t(i))(v_\xi-u_t(\frac{i}{N}))\delta_{\frac{i}{N}}
\end{equation}
We will assume zero Dirichlet boundary conditions for this PDE (clamped axon). We are interested in the process $(u_t,r_t)_{t\in[0,T]}$.

We recall here the result of \cite{Austin} which states that there exists a stochastic process satisfying equations (\ref{dyn}) and (\ref{saut}). Let $u_0$ be in $\mathcal{D}(\Delta)$ such that $\min_{\xi\in E}v_\xi\leq u_0\leq\max_{\xi\in E}v_\xi$, the initial potential of the axon. Let $q_0\in\RR$ be the initial configuration of ion channels.

\begin{proposition}[\cite{Austin}]\label{prop_A1}
Fix $N\geq 1$ and let $(\Omega, \mathcal{F},(\mathcal{F}_t)_{0\leq t\leq T},\PP)$ be a filtered probability space satisfying the usual conditions. There exist a pair $(u_t,r_t)_{0\leq t\leq T}$ of c\`adl\`ag adapted stochastic processes such that each sample path of $u$ is a continuous map from $[0,T]$ to $H$ and $r_t$ is in $\RR$ for all $t\in[0,T]$ and satisfying:
\begin{itemize}
\item(Regularity)  The map $t\mapsto \partial_t u_t$ lies in $L^2([0,T],H^*)$ almost surely.
\item(Dynamic: PDE) 
\begin{equation*}
\partial_t u_t=\Delta u_t +\frac{1}{N}\sum_{\xi\in E}\sum_{i\in \ZZ}c_\xi 1_{\xi}(r_t(i))(v_\xi-u_t(\frac{i}{N}))\delta_{\frac{i}{N}}\end{equation*}
$\forall t\in[0,T]$, $\PP$-almost surely.
\item(Dynamic: jump)
\begin{equation*}
\PP(r_{t+h}(i)=\zeta|r_t(i)=\xi)=\alpha_{\xi,\zeta}\left(u_t\left(\frac{i}{N}\right)\right)h+o(h)
\end{equation*}
\item(Initial condition: PDE) $u_0$ given.
\item(Initial condition: jump) $q_0$ given.
\item(Boundary conditions: PDE only) $u_t(0)=u_t(1)=0$, $\forall t\in[0,T]$.
\end{itemize}
Moreover there exists a constant $C$ such that $\|u_t\|_\infty\leq C$ for all $t\in[0,T]$ and $\omega$.
\end{proposition}

In \cite{Austin} the author proves that when the number of ion channels increases to infinity, the above model converges, in a sense, toward a deterministic model. In our paper, unlike in \cite{Austin} we will work with a fixed number of ion channels but introduce two time scales in the model in order to perform a slow-fast analysis.

\subsection{The spatial stochastic Hodgkin-Huxley model as a Piecewise Deterministic Markov Process in infinite dimensions}

The paper \cite{Buck_Ried} extends the theory of finite dimensional Piecewise Deterministic Markov Processes (PDMP) introduced in \cite{Davis_2} and \cite{Davis_1} to PDMP in infinite dimensions. The results stated here come from section 2.3 and 3.1 of \cite{Buck_Ried} adapted to our particular notations.

For all $r\in\RR$ we define the following function on $H$:
\begin{equation}\label{reac_r}
G_r(u)=\frac{1}{N}\sum_{\xi\in E}\sum_{i\in \ZZ}c_\xi 1_{\xi}(r(i))(v_\xi-u(\frac{i}{N}))\delta_{\frac{i}{N}}
\end{equation}
which is $H^*$ valued. We do not write the dependence in $N$ in the notation of $G_r(u)$ since, unlike in \cite{Austin}, $N$ is a fixed parameter here.

The stochastic process $(u_t,r_t)_{t\in[0,T]}$ can be described in the following way. We start with a given initial potential $v_0$ and ion channels configuration $q_0$. Then the $u$ component evolves according to the evolution equation:
\[
\partial_t u_t=\Delta u_t +G_{q_0}(u_t)
\]
until the first jump of the $r$ component at time $\tau_1$ which occurs according to the transition rate function $\Lambda : H\times\RR\to \R_+$ given for $(u,r)\in H\times\RR$ by:
\[
\Lambda(u,r)=\sum_{i\in\ZZ}\sum_{\xi\neq r(i)}\alpha_{r(i),\xi}(u(\frac iN))
\]
according to (\ref{saut}). A new value $q_1$ for $r$ is then chosen according to a transition measure $Q$ from $H\times\RR$ to $\mathcal{P}(\RR)$ (the set of all probability measures on $\RR$). This transition measure is also given by the jump distribution of the ion channels. For $(u,r)\in H\times\RR$ and $r'\in\RR$ which differs from $r$ only by the component $r(i_0)$:
\[
Q(u,r)(\{r'\})=\frac{\alpha_{r(i_0),r'(i_0)}(u(\frac{i_0}{N}))}{\Lambda(u,r)}
\]
If $r'$ differs from $r$ by two or more components then $Q(u,r)(\{r'\})=0$.

Then  $u$ evolves according to the "updated" evolution equation:
\[
\partial_t u_t=\Delta u_t +G_{q_1}(u_t)
\]
 with initial condition $u_{\tau_1}$ and Dirichlet boundary conditions, until the second jump of the $r$ component. And so on. This description justifies the term "Piecewise Deterministic Process": the dynamic of the process is indeed purely deterministic between the jumps.

As shown in \cite{Buck_Ried} in Theorem 6, the equations of Proposition \ref{prop_A1} define a standard Piecewise Deterministic Process (PDP) in the sense of Definition 3 of \cite{Buck_Ried}. We recall briefly here what this means technically:
\begin{itemize}
\item for a given ion channel configuration $r\in\RR$, the PDE:
\begin{equation*}\label{dyn_r_fix}
\partial_t u_t=\Delta u_t +G_r(u_t)
\end{equation*}
with zero Dirichlet boundary conditions admits a unique global (weak) solution continuous on $H$ for every initial value $u_0\in H$ denoted by $\psi_r(t,x)$.\\
\item The number of state switches in the $(r_t)_{t\in[0,T]}$ component during any finite time interval is finite almost surely.\\
\item The transition rate $\Lambda$ from one state of $\RR$ to another is a measurable function from $H\times\RR$ to $\R_+$ and for all $(u,r)\in H\times\RR$, the transition rate as a function of time is integrable over every finite time interval but diverges to $\infty$ when it is integrated over $\R_+$ (the expected number of jumps tends to $\infty$ when the time horizon increases).\\
\item The transition measure $Q$ from one state of $\RR$ to another is such that the mapping $(u,r)\to Q(u,r)(R)$ is measurable for every $R\subset\RR$ and $Q(u,r)(\{r\})=0$ for all $(u,r)\in H\times\RR$.
\end{itemize}
In fact the stochastic process $(u_t,r_t)_{t\in[0,T]}$ is more than a piecewise deterministic process, it is a piecewise deterministic \emph{Markov} process.
\begin{theorem}\label{th_BR}
\begin{enumerate} \item There exists a filtered probability space satisfying the usual conditions such that our infinite-dimensional standard PDP is a homogeneous Markov process on $H\times\RR$.\\
\item Locally bounded measurable functions on $H\times\RR$ which are absolutely continuous as maps $t\mapsto f(\psi_r(t,x),r)$ for all (x,r) are in the domain $\mathcal{D}(\mathcal{A})$ of the extended generator. The extended generator is given for almost all $t$ by:
\begin{equation}\label{gene}
\mathcal{A}f(u_t,r_t)=\frac{d}{dt}f(u_t,r_t)+\Ba f(u_t,r_t)
\end{equation}
where
\[
\Ba f(u_t,r_t)=\sum_{i\in\ZZ}\sum_{\zeta\in E}[f(u_t,r_t(r_t(i)\to\zeta))-f(u_t,r_t)]\alpha_{r_t(i),\zeta}(u_t(\frac{i}{N}))
\]
with $r_t(r_t(i)\to\zeta)$ is the component of $\RR$ with $r_t(r_t(i)\to\zeta)(j)$ equals to $r_t(j)$ if $j\neq i$ and to $\zeta$ if $j=i$. The notation $\frac{d}{dt}f(u_t,r_t)$ means that the function $s\mapsto f(u'_s,r)$ is differentiated at $s=t$, where $u'$ is the solution of the PDE of Proposition \ref{prop_A1} such that $u'_t=u_t$ and with the channel state $r$ held fixed equal to $r_t$.
\end{enumerate}
\end{theorem}

\begin{remark}
It is not usual to write a generator only along paths $(u_t,r_t)_{t\geq0}$ as we do. We would expect for a generator an analytical expression of the form $\mathcal{A}f(u,r)$ for any $(u,r)\in H\times\RR$. It is in fact possible to obtain such an expression for more regular function $f$. The part of the generator related to the continuous Markov chain $r$ takes the form $\Ba f(u,r)$ with no changes but the derivative $\frac{d}{dt}f(u,r)$ can then be expressed thanks to the Fr\'echet derivative of $f$, see \cite{Buck_Ried} for more details. We will not use this refinement in the sequel.
\end{remark}

Knowing that a stochastic process is Markovian is crucial for its study since a lot of mathematical tools have been developed for Markovian processes, see for example \cite{Eth_Kur}.

\subsection{Singularly perturbed model and main results}

\subsubsection{Introduction of two time scales in the model.}

We introduce now two time scales in the model. Indeed, we consider that the Hodgkin-Huxley model is a slow-fast system: some states of the ion channels communicate faster between them than others, see for example \cite{Hille}. Mathematically, this leads to introduce a small parameter $\e>0$ in our equation. For the states which communicate at a faster rate, we say that they communicate at the usual rate divided by $\e$.

We can then consider different classes of states or partition the state space $E$ in states which communicate at a high rate. This kind of description is very classical, see for example \cite{Fa_Cri}. We regroup our states in classes making a partition of the state space $E$ into: 
\[
E=E_1\sqcup\cdots\sqcup E_l
\]
where $l\in\{1,2,\cdots\}$ is the number of classes. Inside a class $E_j$, the states communicate faster at jump rates of order $\frac1\e$. States of different classes communicate at the usual rate of order $1$. For $\e>0$ fixed, we denote by $(u^\e,r^\e)$ the modification of the PDMP introduced in the previous section with now two time scales. Its generator is, for $f\in\mathcal{D}(\mathcal{A})$:
\begin{equation}\label{gene_eps}
\mathcal{A}f(u^\e_t,r^\e_t)=\frac{d}{dt}f(u^\e_t,r^\e_t)+\Ba^\e f(u^\e_t,r^\e_t)
\end{equation}
$\Ba^\e$ is the component of the generator related to the continuous time Markov chain $r^\e$. According to our slow-fast description we have the two time scales decomposition of this generator:
\begin{equation*}\label{gene_saut}
\Ba^\e =\frac1\e\Ba+\hat\Ba
\end{equation*}
where the "fast" generator $\Ba$ is given by:
\begin{equation*}\label{gene_fast}
\Ba f(u^\e_t,r^\e_t)=\sum_{i\in\ZZ}\sum_{j=1}^l1_{E_j}(r^\e_t(i))\sum_{\zeta\in E_j}[f(u^\e_t,r^\e_t(r^\e_t(i)\to\zeta)-f(u^\e_t,r^\e_t)]\alpha_{r^\e_t(i),\zeta}(u^\e_t(\frac{i}{N}))
\end{equation*}
and the "slow" generator is given by:
\begin{equation*}\label{gene_slow}
\hat\Ba f(u^\e_t,r^\e_t)=\sum_{i\in\ZZ}\sum_{j=1}^l1_{E_j}(r^\e_t(i))\sum_{\zeta\notin E_j}[f(u^\e_t,r^\e_t(r^\e_t(i)\to\zeta)-f(u^\e_t,r^\e_t)]\alpha_{r^\e_t(i),\zeta}(u^\e_t(\frac{i}{N}))
\end{equation*}

For $y\in \R$ fixed and $g:\R\times E\to\R$, we denote by $\Ba_j(y)$, $j\in\{1,\cdots,l\}$ the following generator:
\[
\Ba_j(y)g(\xi)=1_{E_j}(\xi)\sum_{\zeta\in E_j}[g(y,\zeta)-g(y,\xi)]\alpha_{\xi,\zeta}(y)
\]
which will be of interest in the sequel.

\subsubsection{Uniform boundedness.}

Here is a crucial result for the proof of our main result. Its proof takes back the argument developed in \cite{Austin} but for the sake of completeness and given the intensive use of this proposition in the sequel, we provide a short proof.
\begin{proposition}\label{Prop_bound}
For any $T>0$, there is a deterministic constant $C>0$ independent of $\e\in]0,1]$ such that:
\begin{equation*}
\sup_{t\in[0,T]} \|u^\e_t\|_H\leq C
\end{equation*}
almost surely.
\end{proposition}
In the proof of this proposition, we will use the following representation of a solution of the PDE part of our PDMP. We say that $u^\e$ is a mild solution to (\ref{dyn}) if:
\begin{equation}\label{mild}
u^\e_t=e^{\D t}u_0+\int_0^t e^{\D (t-s)}G_{r^\e_s}(u^\e_s)ds,\quad t\geq 0
\end{equation}
almost surely where $(e^{\D t},t\geq0)$ is the semi-group associated to $\D$ in $H$. For any $u\in H$:
\begin{equation*}\label{semig}
e^{\D t}u=\sum_{k\geq 1}e^{-(k\pi)^2t}(u,e_k)e_k,\quad t>0
\end{equation*}
Let us mention that the mild formulation (\ref{mild}) holds also up to a bounded stopping time $\tau$, this will be useful later in the paper. 
\begin{proof}
We work $\omega$ by $\omega$. Using the mild formulation (\ref{mild}), we see that:
\begin{eqnarray*}
\|u^\e_t\|_H&\leq& \|e^{\Delta t}u_0\|_H+\frac{1}{N}\sum_{\xi\in E}\sum_{i\in \ZZ}\|\int_0^tc_\xi 1_{\xi}(r^\e_s(i))v_\xi e^{\Delta(t-s)}\delta_{\frac{i}{N}}ds\|_H\\
&+&\frac{1}{N}\sum_{\xi\in E}\sum_{i\in \ZZ}\|\int_0^tc_\xi u^\e_s(\frac{i}{N}) e^{\Delta(t-s)}\delta_{\frac{i}{N}}ds\|_H
\end{eqnarray*}
By Lemma \ref{Aulem} of Appendix \ref{app_lem} the first of the above sums is bounded by a fixed constant which can be chosen independent of $t\in[0,T]$. For the second sum, we break the integral in two pieces. For all $\gamma>0$ by Lemma \ref{Aulem}, we can choose $\eta>0$ so small that:
\[
\|\int_{t-\eta}^tc_\xi u^\e_s(\frac{i}{N}) e^{\Delta(t-s)}\delta_{\frac{i}{N}}ds\|_H\leq \frac{1}{2}\gamma\max_{s\in[0,t]}|c_\xi u^\e_s(\frac iN)|
\]
Since $|c_\xi u^\e_s(\frac iN)|\leq \max c_\xi \|u^\e_s\|_\infty\leq \max c_\xi C_P\|u^\e_s\|_H$, choosing $\gamma$ such that $\gamma\max c_\xi C_P<1$ we have:
\[
\|\int_{t-\eta}^tc_\xi u^\e_s(\frac{i}{N}) e^{\Delta(t-s)}\delta_{\frac{i}{N}}ds\|_H\leq \frac{1}{2}\max_{s\in[0,t]}\|u^\e_s\|_H
\]
Now the second estimates of Lemma \ref{Aulem}  gives: 
\begin{eqnarray*}
\|\int_0^{t-\eta}c_\xi u^\e_s(\frac{i}{N}) e^{\Delta(t-s)}\delta_{\frac{i}{N}}ds\|_H&\leq& C_2(\eta)\max_{\xi\in E}|c_\xi|\int_0^{t-\eta}\|u^\e_s\|_\infty ds\\
&\leq& C_PC_2(\eta)\max_{\xi\in E}|c_\xi|\int_0^{t-\eta}\max_{\sigma\in[0,s]}\|u_\sigma^\e\|_H ds
\end{eqnarray*}
Reassembling the above inequalities, the end of the proof is a classical application of the Gronwall's lemma. We find finally that  $\max_{s\in[0,t]}\|u_t^\e\|_H$ is bounded by a fixed constant independent of $t\in[0,T]$, $\omega$ and $\e\in]0,1]$. 
\end{proof}

\subsubsection{Main assumptions and averaged equation.}\label{MAAE}

We make the following assumptions. For any $y\in \R$ fixed, and any $j\in\{1,\cdots,l\}$, the fast generator $\Ba_j(y)$ is weakly irreducible on $E_j$, i.e has a unique quasi-stationary distribution denoted by $\mu_j(y)$.  This quasi-stationary distribution is supposed to be, as is its derivative, Lipschitz-continuous in its argument $y$ (see section \ref{Campaign} for more details).

Following \cite{Yin_Zhang}, the states in $E_j$ can be considered as equivalent. For any $i\in\ZZ$ we define the new stochastic process $(\bar r_t^\epsilon)_{t\geq0}$ by $\bar r_t^\e(i)=j$ when $ r_t^\e(i)\in E_j$ and abbreviate $E_j$ by $j$ . We then have an aggregate process $\bar r^\e(i)$ with values in $\{1,\cdots,l\}$. It is not a Markov process but we have the following proposition:

\begin{proposition}[\cite{Yin_Zhang}]\label{YZ}
For any $y\in\R$, $i\in\ZZ$ and $g:\{1,\cdots,l\}\to\R$, $\bar r^\e(i)$ converges weakly when $\e$ goes to $0$ to a Markov process $\bar r(i)$ generated by:
\[
\bar\Ba(y) g(\bar r(i))=\sum_{j=1}^l 1_{j}(\bar r(i))\sum_{k=1,k\neq j}^l (g(k)-g(j))\sum_{\xi\in E_k}\sum_{\zeta\in E_j}\alpha_{\zeta,\xi}(y)\mu_j(\zeta)
\]
\end{proposition}

We then average the function $G_r(u)$ against the quasi-invariant distributions. That is we consider that each pack of states $E_j$ has reached its stationary behavior. For any $\bar r\in\bar\RR=\{1,\cdots,l\}^{\ZZ}$ we define the averaged function by:
\begin{equation}\label{F}
F_{\bar r}(u)=\frac{1}{N}\sum_{i\in \ZZ}\sum_{j=1}^l 1_{j}(\bar r(i))\sum_{\zeta\in E_j}c_\zeta \mu_j(u(\frac{i}{N}))(\zeta)(v_\zeta-u(\frac{i}{N}))\delta_{\frac{i}{N}}
\end{equation}

Therefore, we call averaged equation of (\ref{dyn}) the following PDE:
\begin{equation}\label{Au_moy}
\partial_t u_t=\D u_t + F_{\bar r_t}(u_t)
\end{equation}
with zero Dirichlet boundary condition and initial condition $u_0$ and $\bar q_0$ (where $\bar q_0$ is the aggregation of the initial channel configuration $q_0$). In equation (\ref{Au_moy}), the process $(\bar r_t)_{t\in[0,T]}$ evolves, each coordinate independently over infinitesimal time-scales, according to the averaged jump-rates between the subsets $E_j$ of $E$.

\begin{proposition}
For any $T>0$ there exists a probability space satisfying the usual conditions such that equation (\ref{Au_moy}) defines a PDMP $(u_t,\bar{r}_t)_{t\in[0,T]}$ in infinite dimensions in the sense of \cite{Buck_Ried}. Moreover, there is a constant $C$ such that:
\begin{equation*}
\sup_{t\in[0,T]}\|u_t\|_H\leq C
\end{equation*}
and $u\in\mathcal{C}([0,T],H)$ almost-surely.
\end{proposition}
\begin{proof}
The equation (\ref{Au_moy}) is of the form (\ref{dyn}) except that the indicator function is replaced by the probabilities $\mu_j$. This changes nothing to the mathematics and the arguments developed in \cite{Austin} or \cite{Buck_Ried} still apply. We refer the reader to \cite{Austin} and \cite{Buck_Ried} for more details.
\end{proof}

We can now state our main result which states that our approximation (\ref{Au_moy}) is valid.

\begin{theorem}\label{APDE}
The stochastic process $u^\e$ solution of (\ref{dyn}) converges in law to a solution of (\ref{Au_moy}) when $\e$ goes to $0$.
\end{theorem}

More precisely, we will prove the following proposition:

\begin{proposition}\label{prop_main}
$(u^\e,\e\in]0,1])$ is tight in $\mathcal{C}([0,T],H)$ and any accumulation point $u$ verifies:
\begin{equation}\label{id_lim}
(u_t,\phi)_{L^2(I)}=(u_0,\phi)_{L^2(I)}+\int_0^t (u_s,\phi'')_{L^2(I)}ds+\int_0^t <F_{\bar{r}_s}(u_s),\phi>ds
\end{equation}
for all $\phi$ in $\mathcal{C}^2_0(I)$. Moreover the accumulation point is unique up to indistinguishability.  
\end{proposition}

\subsubsection{Plan of the campaign.}\label{Campaign}

To make the proof of our main result easier to read, we will consider in a first step that all the states communicate at a fast rate. This is called the all-fast case. We will prove our main result in detail in this case and then give the key points for the validity of our proof in the general case. In the all-fast case there is a unique class $E$ and the generator has the simplest form:
\begin{equation}\label{gene}
\Ba f(u,r)=\sum_{i\in\ZZ}\sum_{\zeta\in E}[f(u,r(r(i)\to\zeta))-f(u,r)]\alpha_{r(i),\zeta}(u(\frac{i}{N}))
\end{equation}
We make the following assumptions: for any $i\in\ZZ$ with $u\in H$ fixed, the Markov process $r(i)$ has a unique stationary distribution $\mu(u(\frac{i}{N}))$. Then the process $(r(i),i\in\ZZ)$ has the following stationary distribution:
\begin{equation*}
\mu(u)=\bigotimes_{i\in\ZZ} \mu(u(\frac{i}{N}))
\end{equation*}
We assume that the rate functions are Lipschitz continuous. The average (generalized) function is then:
\begin{eqnarray}\label{F}
F(u)&=&\int_{\RR}G_r(u)\mu(u)(dr)\\
&=&\frac{1}{N}\sum_{\xi\in E}\sum_{i\in \ZZ}c_\xi \mu(u(\frac{i}{N}))(\xi)(v_\xi-u(\frac{i}{N}))\delta_{\frac{i}{N}}\nonumber
\end{eqnarray}
if $u$ is held fixed.

Is this hypothesis reasonable?  For the most classical model in neurosciences, the rate functions are indeed, as well as their derivatives, Lipschitz continuous. See for example the rate functions in classical Hodgkin-Huxley model \cite{HH} (see also the Appendix \ref{app}).\\
\\
Let us show briefly that in this case the ergodic measure $\mu$ is Lipschitz continuous in its argument. Let $i$ be in $\ZZ$ and the membrane potential at $\frac iN$ be fixed equal to $x\in\R$, the state space of the continuous time Markov chain $r(i)$ is $E$ which is a finite set. We set $|E|=m$. We can then enumerate the elements of $E$ : $E=\{\xi_1,\cdots,\xi_m\}$. The generator of $r(i)$ is then the matrix :
\[
J(x)=(\alpha_{\xi_i,\xi_j}(x))_{1\leq i,j\leq m}
\]
with $\alpha_{\xi_i,\xi_i}(x)=-\sum_{j\neq i} \alpha_{\xi_i,\xi_j}(x)$. Therefore, if we assumed that our continuous time Markov chain is ergodic, it has an invariant measure which is of the form:
\[
\mu(x)(\xi)=\frac{\sum \prod\alpha_{\xi,\xi'}(x)}{\sum \prod\alpha_{\zeta,\zeta'}(x)}
\]
The sums and products involved here are calculated over subsets of $E$ which are given by resolving the equation $\mu(x)^TJ=0$. $\mu$ is then Lipschitz continuous in $x$ since each $\alpha_{\xi,\zeta}$ is Lipschitz continuous in $x$ and bounded. 

The assumption of independence of each coordinate of the process $r$ over infinitesimal time-scales is important to have a simple expression of the invariant measure of the process $r$ and simplify the mathematics. It is possibly unrealistic from the biological point of view but is often assumed in mathematical neurosciences. 

The plan of the campaign to prove our main result (\ref{APDE}) in the all-fast case is the following. We want to prove the convergence in law of a family of c\`adl\`ag stochastic processes with values in a Hilbert space towards another, here deterministic, process. There is a well known strategy to do that:
\begin{itemize}
\item prove the tightness of the family
\item identify the limit
\end{itemize}
In our case, tightness (see subsection \ref{main_proof}) will follow roughly from the uniform boundedness of the family $(u^\e,\e\in]0,1])$ obtained in Proposition \ref{Prop_bound}. The identification of the limit will follow on one side from tightness and on the other side from some classical argument in the theory of averaging for stochastic processes, see for example \cite{PS}. We will introduce a Poisson equation which will enable us to control some problematic terms (see Proposition \ref{Prop_Poi}) and identify the limit (subsection \ref{main_proof}). We will then move to the general case with multiple classes (in  subsection \ref{Cas_g}) with an emphasis on each point which could give an issue and show how things work in this case.

\section{Proof of the main result}\label{section_proof}

\subsection{The all-fast case}\label{main_proof}

We want to prove Proposition \ref{prop_main}. Let us recall that in the all-fast case $F_{\bar{r}}$ in (\ref{id_lim}) reduces to $F$ in (\ref{F}). 

Let us outline the strategy of the proof following the plan of the campaign announced in Section \ref{Campaign}.
\begin{itemize}
\item Tightness.
\begin{enumerate}
\item Thanks to Proposition \ref{Prop_bound}, use the Markov inequality to prove that Aldous's condition and property (\ref{TH1}) of Theorems \ref{GCT} and \ref{TH} in Appendix \ref{Met} hold.
\item For $t\in[0,T]$ fixed, find the finite dimensional space $L_{\delta,\eta}$ of (\ref{TH2}), Theorem \ref{TH} Appendix \ref{Met} using a truncation of the solution $u^\e_t$ in the mild formulation. Bound $ \EE(d(u^\e_t,L_{\delta,\eta})>\eta)$ and use the Markov inequality.
\item Obtain tightness in $\mathbb{D}([0,T],H)$ by Theorems \ref{GCT} and \ref{TH} Appendix \ref{Met} and conclude that tightness in $\mathcal{C}([0,T],H)$ holds as well thanks to the regularity of the solution $u^\e$ for any $\e\in]0,1]$.
\end{enumerate}
\item  Identification of the limit.
\begin{enumerate}
\item Begin by showing that the following property is sufficient:
\[
\lim_{\e\to0}\EE\left|\int_0^t<G_{r^\e_s}(u^\e_s)-F(u^\e_s),\phi>ds\right|=0 \quad\forall\phi\in\mathcal{C}^2_0(I)
\]
To prove this property, use the tightness of the family $(u^\e,\e\in]0,1])$.
\item Replace the term $<G_{r^\e_s}(u^\e_s)-F(u^\e_s),\phi>$ by $\Ba f(u^\e_s,r^\e_s)$ where  $f$ is the solution of a Poisson equation. Show that this solution exists and has some nice properties.
\item Write the semi-martingale expansion of $\int_0^t<G_{r^\e_s}(u^\e_s)-F(u^\e_s),\phi>ds$. Bound its martingale part thanks to the bound for the bracket of the martingale.
\item Bound its finite variation part thanks to the nice properties of $f$ and the uniform bound  from Proposition \ref{Prop_bound}.
\item Aggregate all these bounds together to obtain that  $\EE\left|\int_0^t<G_{r^\e_s}(u^\e_s)-F(u^\e_s),\phi>ds\right|^2$ is bounded by a constant times $\e$.
\item Conclude that (\ref{id_lim}) holds and implies uniqueness up to indistinguishability.
\end{enumerate}
\end{itemize} 

\begin{proof}[Tightness]
STEP 1 (Property  (\ref{TH1}) and Aldous condition from Appendix \ref{Met}). We begin with the Aldous condition. Let $\tau$ be a stopping time and $\theta$ such that $\tau+\theta<T$. We need to control the size of $\|u^\e_{\tau+\theta}-u^\e_\tau\|_H$. At first we work deterministically. We break the difference $\|u^\e_{\tau+\theta}-u^\e_\tau\|^2_H$ into three pieces using the mild formulation of $u^\e$ in (\ref{mild}):
\begin{eqnarray*}
&~&\|u^\e_{\tau+\theta}-u^\e_\tau\|^2_H\\
&\leq&4\|e^{\Delta(\tau+\theta)}u_0-e^{\Delta\tau}u_0\|^2_H+4\|\int_\tau^{\tau+\theta}e^{\Delta(\tau+\theta-s)}G_{r^\e_s}(u^\e_s)ds\|^2_H\\
&+&4\|\int_0^\tau \left(e^{\Delta(\tau+\theta-s)}-e^{\Delta(\tau-s)}\right)G_{r^\e_s}(u^\e_s)ds\|^2_H
\end{eqnarray*}
Let us show now that the supremum over $\theta\in]0,\eta[$ of each term on the right hand side of the above inequality goes to $0$ with $\eta$  (we take $\eta>0$). We start with the third term:
\begin{eqnarray*}
&~&\|\int_0^{\tau}\left(e^{\Delta(\tau+\theta-s)}-e^{\Delta(\tau-s)}\right)G_{r^\e_s}(u^\e_s)ds\|^2_H
\end{eqnarray*}
Notice that:
\begin{eqnarray*}
&~&\left(e^{\Delta(\tau+\theta-s)}-e^{\Delta(\tau-s)}\right)G_{r^\e_s}(u^\e_s)\\
&=&\sum_{k\geq1}e^{-(k\pi)^2(\tau-s)}\left(e^{-(k\pi)^2\theta}-1\right)\frac{1}{N}\sum_{\xi\in E}\sum_{i\in \ZZ}c_\xi 1_{\xi}(r^\e_s(i))(v_\xi-u^\e_s(\frac{i}{N}))(1+(k\pi)^2)e_k\left({\frac{i}{N}}\right)e_k
\end{eqnarray*}
where we have used the fact that $<e_k,\delta_{\frac iN}>=(1+(k\pi)^2)e_k(\frac iN)$ ( recall that $e_k=\frac{\sqrt2}{\sqrt{1+(k\pi)^2}}\sin(k\pi\cdot)$). Thus we obtain:
\begin{eqnarray*}
&~&\|\int_0^{\tau}\left(e^{\Delta(\tau+\theta-s)}-e^{\Delta(\tau-s)}\right)G_{r^\e_s}(u^\e_s)ds\|^2_H\\
&\leq&4|E|^2\max_{\xi\in E} c^2_\xi(\max_{\xi\in E} v^2_\xi+C^2)\sum_{k\geq1}\left(\int_0^{\tau}e^{-(k\pi)^2(\tau-s)}\left(e^{-(k\pi)^2\theta}-1\right)\sqrt{1+(k\pi)^2}ds\right)^2\\
&=&4|E|^2\max_{\xi\in E} c^2_\xi(\max_{\xi\in E} v^2_\xi+C^2)\sum_{k\geq1}\left(e^{-(k\pi)^2\theta}-1\right)^2(1+(k\pi)^2)\frac{\left(1-e^{-(k\pi)^2\tau}\right)^2}{(k\pi)^4}\\
&\leq&4|E|^2\max_{\xi\in E} c^2_\xi(\max_{\xi\in E} v^2_\xi+C^2)\sum_{k\geq1}\left(e^{-(k\pi)^2\theta}-1\right)^2\frac{1+(k\pi)^2}{(k\pi)^4}
\end{eqnarray*}
where $C$ is the constant of Proposition \ref{Prop_bound}. By dominated convergence, we see that:
\[
\lim_{\eta\to0}\sup_{\theta\in]0,\eta[}\sum_{k\geq1}\left(e^{-(k\pi)^2\theta}-1\right)^2\frac{1+(k\pi)^2}{(k\pi)^4}\leq\lim_{\eta\to0}\sum_{k\geq1}\left(e^{-(k\pi)^2\eta}-1\right)^2\frac{1+(k\pi)^2}{(k\pi)^4}=0
\]
For the second term as before, we can show that:
\begin{eqnarray*}
&~&\|\int_\tau^{\tau+\theta}e^{\Delta(\tau+\theta-s)}G_{r^\e_s}(u^\e_s)ds\|^2_H\\
&\leq&4|E|^2\max_{\xi\in E} c^2_\xi(\max_{\xi\in E} v^2_\xi+C^2)\sum_{k\geq1}\left(\int_\tau^{\tau+\theta}e^{-(k\pi)^2(\tau+\theta-s)}\sqrt{1+(k\pi)^2}ds\right)^2\\
&\leq&4|E|^2\max_{\xi\in E} c^2_\xi(\max_{\xi\in E} v^2_\xi+C^2)\sum_{k\geq1}(1+(k\pi)^2)\frac{\left(1-e^{-(k\pi)^2\theta}\right)^2}{(k\pi)^4}
\end{eqnarray*}
By dominated convergence, we see that:
\[
\lim_{\eta\to0}\sup_{\theta\in]0,\eta[}\sum_{k\geq1}(1+(k\pi)^2)\frac{\left(1-e^{-(k\pi)^2\theta}\right)^2}{(k\pi)^4}\leq\lim_{\eta\to0}\sum_{k\geq1}(1+(k\pi)^2)\frac{\left(1-e^{-(k\pi)^2\eta}\right)^2}{(k\pi)^4}=0
\]
For the first term we have, by the Bessel-Parseval equality:
\begin{eqnarray*}
&~&\|e^{\Delta(\tau+\theta)}u_0-e^{\Delta\tau}u_0\|^2_H\\
&=&\|\sum_{k\geq1}e^{-(k\pi)^2\tau}(e^{-(k\pi)^2\theta}-1)(u_0,e_k)e_k\|^2_H\\
&=&\sum_{k\geq1}e^{-2(k\pi)^2\tau}(e^{-(k\pi)^2\theta}-1)^2(u_0,e_k)^2\\
&\leq&\sum_{k\geq1}(e^{-(k\pi)^2\theta}-1)^2(u_0,e_k)^2
\end{eqnarray*}
By dominated convergence, we see that:
\[
\lim_{\eta\to0}\sup_{\theta\in]0,\eta[}\sum_{k\geq1}(e^{-(k\pi)^2\theta}-1)^2(u_0,e_k)^2\leq\lim_{\eta\to0}\sum_{k\geq1}(e^{-(k\pi)^2\eta}-1)^2(u_0,e_k)^2=0
\]
Combining the results obtained for the three terms we have:
\[
\lim_{\eta\to0}\sup_{\theta\in]0,\eta[}\|u^\e_{\tau+\theta}-u^\e_\tau\|^2_H=0
\]
uniformly in $\e\in]0,1]$. Therefore, for all $M,\delta>0$, using the Chebyshev inequality, we can choose $\eta$ so small that:
\[
\sup_{\e\in]0,1]}\sup_{\theta\in]0,\eta[}\PP(\|u^\e_{\tau+\theta}-u^\e_\tau\|_H\geq M)\leq\sup_{\e\in]0,1]}\sup_{\theta\in]0,\eta[}\frac{\EE(\|u^\e_{\tau+\theta}-u^\e_\tau\|^2_H)}{M^2}<\delta
\]

The first condition (\ref{TH1}), Theorem \ref{TH}  Appendix \ref{Met} is verified by the application of the Markov inequality. Indeed, for any $t\in[0,T]$ and $\delta>0$ by Proposition \ref{Prop_bound} there exists a constant $C>0$ independent of $\e\in]0,1]$ and $t\in[0,T]$ such that:
\[
 \sup_{\e\in]0,1]} \PP(\|u^\e_t\|_{H}>\rho)\leq \frac{1}{\rho}\sup_{\e\in]0,1]}\EE\|u^\e_t\|_{H}\leq\frac{C}{\rho}<\delta
\]
for any $\rho>0$ large enough.

STEP 2 (Truncation). We have to show that for any $\delta,\eta>0$ we can find $\e_0>0$ and a space $L_{\delta,\eta}$ such that:
\[
\sup_{\e\in]0,\e_0]} \PP(d(u^\e_t,L_{\delta,\eta})>\eta)\leq\delta
\]
where $d(u^\e_t,L_{\delta,\eta})=\inf_{v\in L_{\delta,\eta}}\|u^\e_t-v\|_H$. We fix $t\in[0,T]$. Recalling the mild representation (\ref{mild}), we have:
\[
u^\e_t=e^{\Delta t}u_0+\int_0^t\frac{1}{N}\sum_{\xi\in E}\sum_{i\in \ZZ}c_\xi 1_{\xi}(r^\e_s(i))(v_\xi-u^\e_s(\frac{i}{N}))e^{\Delta(t-s)}\delta_{\frac{i}{N}}ds
\]
where, using the explicit expression of the semi-group $e^{\D t}$, $e^{\Delta(t-s)}\delta_{\frac{i}{N}}$ is equal to $\sum_{k\geq1}e^{-(k\pi)^2(t-s)}(1+(k\pi)^2)e_k(\frac{i}{N})e_k$. We define, for $f\in H$ and $x\in I$:
\begin{eqnarray*}
e^{\Delta t}_pf&:=&\sum_{k=1}^pe^{-(k\pi)^2t}(f,e_k)e_k\text{ and } e^{\Delta t}_p\delta_x:=\sum_{k=1}^pe^{-(k\pi)^2t}(1+(k\pi)^2)e_k(x)e_k,\\
u^{\e,p}_t&:=&e^{\Delta t}_pu_0+\int_0^t\frac{1}{N}\sum_{\xi\in E}\sum_{i\in \ZZ}c_\xi 1_{\xi}(r^\e_s(i))(v_\xi-u^\e_s(\frac{i}{N}))e^{\Delta(t-s)}_p\delta_{\frac{i}{N}}ds
\end{eqnarray*}
For any $\delta,\eta>0$ we can find $p$ independent of $\e$ such that $\|u^\e_t-u^{\e,p}_t\|_H\leq C'\eta\delta$ with the constant $C'$ deterministic and independent of $\delta,\eta$ and $\e\in]0,1]$. Indeed we can easily show, as in STEP 1, that:
\[
\|u^\e_t-u^{\e,p}_t\|^2_H\leq 2\|\sum_{k=p+1}^\infty e^{-(k\pi)^2t}(u_0,e_k)e_k\|^2_H+2C'\sum_{k=p+1}^\infty (1+(k\pi)^2)\frac{(1-e^{-(k\pi)^2t})^2}{(k\pi)^4}\\
\]
which is independent of $\e\in]0,1]$. The convergence of each series, uniformly in $\e$, enables us to choose a suitable $p$ independent of $\e$. Let us denote $L_p=span\{e_i,1\leq i\leq p\}$. We choose $ L_{\delta,\eta}=L_p$. Since $u^{\e,p}_t\in L_{\delta,\eta}=L_p$
\[
\EE(d(u^\e_t,L_{\delta,\eta}))\leq\EE(\|u^\e_t-u^{\e,p}_t\|_H)\leq C'\eta\delta
\]
Markov's inequality gives us :
\[
\PP(d(u^\e_t,L_{\delta,\eta})>\eta)\leq C'\delta
\]
with $\delta$ independent of $\e\in]0,1]$.

STEP 3 (Tightness). Therefore using Theorems \ref{GCT} and \ref{TH} of Appendix \ref{Met}, $(u^\e,\e\in]0,1])$ is tight in $\mathbb{D}([0,T],H)$. Since we know that for each $\e\in]0,1]$, $u^\e$ is in $\mathcal{C}([0,T],H)$ we have in fact that $(u^\e,\e\in]0,1])$ is tight in $\mathcal{C}([0,T],H)$.

\end{proof}
\begin{proof}[Identification of the limit]
STEP 1 (Reduction of the problem). We know that the family $(u^\e,\e\in]0,1])$ is tight in $\mathcal{C}([0,T],H)$. We still denote by $u^\e$ a converging sub-sequence of $u^\e$ and we denote by $u$ the corresponding accumulation point. We want to show that for all $\phi$ in $\mathcal{C}^2_0(I)$:
\[
(u_t,\phi)_{L^2(I)}=(u_0,\phi)_{L^2(I)}+\int_0^t (u_s,\phi'')_{L^2(I)} ds +\int_0^t<F(u_s),\phi>ds
\]
almost surely. We will show that:
\begin{equation}\label{id_weak}
\EE\left|(u_t,\phi)_{L^2(I)}-(u_0,\phi)_{L^2(I)}-\int_0^t (u_s,\phi'')_{L^2(I)} ds -\int_0^t<F(u_s),\phi>ds\right|=0
\end{equation}
Notice that:
\begin{eqnarray*}
&~&\EE\left|(u^\e_t,\phi)_{L^2(I)}-(u_0,\phi)_{L^2(I)}-\int_0^t (u^\e_s,\phi'')_{L^2(I)} ds -\int_0^t<F(u^\e_s),\phi>ds\right|\\
&=&\EE\left|\int_0^t<G_{r^\e_s}(u^\e_s)-F(u^\e_s),\phi>ds\right|
\end{eqnarray*}
Since the function:
\[
h(v)=\left|(v_t,\phi)_{L^2(I)}-(v_0,\phi)_{L^2(I)}-\int_0^t (v_s,\phi'')_{L^2(I)} ds -\int_0^t<F(v_s),\phi>ds\right|
\]
is continuous on $\mathcal{C}([0,T],H)$ and the arguments $(u^\e,\e\in]0,1])$ and $u$ are uniformly bounded in $\e$ in $\mathcal{C}([0,T],H)$, by convergence in law we have that:
\begin{eqnarray*}
&\lim_{\e\to0}&\EE\left|(u^\e_t,\phi)_{L^2(I)}-(u_0,\phi)_{L^2(I)}-\int_0^t (u^\e_s,\phi'')_{L^2(I)} ds -\int_0^t<F(u^\e_s),\phi>ds\right|\\
&=&\EE\left|(u_t,\phi)_{L^2(I)}-(u_0,\phi)_{L^2(I)}-\int_0^t (u_s,\phi'')_{L^2(I)} ds -\int_0^t<F(u_s),\phi>ds\right|
\end{eqnarray*}
Therefore (\ref{id_weak}) will follow if $\lim_{\e\to0}\EE\left|\int_0^t<G_{r^\e_s}(u^\e_s)-F(u^\e_s),\phi>ds\right|=0$. We will show more precisely that:
\[
\lim_{\e\to0}\EE\left| \int_0^t<G_{r^\e_s}(u^\e_s)-F(u^\e_s),\phi>ds\right|^2=0
\]
 
STEP 2 (Poisson equation). We introduce the following Poisson equation on $f : H\times\RR\to\R$:
\begin{eqnarray}\label{Poisson}
\mathcal{B}f(u,r)&=&<G_{r}(u)-F(u),\phi>\\
\int_\mathcal{R} f(u,r)\mu(u)(dr)&=&0\nonumber
\end{eqnarray}
with $\mathcal{B}$ the generator given by equation (\ref{gene}):
\begin{equation*}
\Ba f(u,r)=\sum_{i\in\ZZ}\sum_{\zeta\in E}[f(u,r(r(i)\to\zeta))-f(u,r)]\alpha_{r(i),\zeta}(u(\frac{i}{N}))
\end{equation*}
\begin{proposition}\label{Prop_Poi}
The Poisson equation (\ref{Poisson}) has a unique solution $f$ which is measurable and locally Lipschitz continuous in its first variable with respect to the $\|\cdot\|_H$ norm. For all fixed $r\in\RR$, the map $u\in H\mapsto f(u,r)$ has a Fr\'echet derivative denoted by $\frac{df}{du}[u,r]$. Moreover $\sup_{s\in[0,T]}|f(u^\e_s,r^\e_s)|$ and $\sup_{s\in[0,T]}\left\|\frac{df}{du}[u^\e_s,r^\e_s]\right\|_{H^*}$ are both bounded almost surely by  constants independent of $\e$.
\end{proposition}

We begin the proof of the above Proposition by the following Lemma.
\begin{lemma}\label{lem_Poi}
Let $B\subset H$ be a bounded domain and $r$ be in $\RR$. The map $u\in H\mapsto <G_{r}(u)-F(u),\phi>$ is:
\begin{enumerate}
\item bounded on $B$ by a constant independent of $r$.
\item Lipschitz continuous on $B$ with Lipschitz constant independent of $r$.
\item Fr\'echet differentiable on $B$ with Fr\'echet derivative in $u\in H$ denoted by $\frac{df}{du}[u,r]$. Moreover the Fr\'echet derivative is bounded in $\|\cdot\|_{H^*}$ norm uniformly in $u\in B$ and $r\in\RR$. 
\end{enumerate}
\end{lemma}
\begin{proof}
Recall that:
\[
<G_{r}(u)-F(u),\phi>=\frac{1}{N}\sum_{\xi\in E}\sum_{i\in \ZZ}c_\xi (1_{\xi}(r(i))-\mu(u(\frac iN))(\xi))(v_\xi-u(\frac{i}{N}))\phi(\frac{i}{N})
\]
Since for $\xi\in E$, $\mu(\cdot)(\xi)$ is bounded and Lipschitz continuous on $\R$, the two first points follow. For the third point we note that the map $u\in H\mapsto <G_{r}(u)-F(u),\phi>$ is a linear combination of the Fr\'echet differentiable functions $u\in H\mapsto \mu\left(u\left(\frac iN\right)\right)(\xi)$, $u\mapsto u\left(\frac iN\right)$ and of the product of these two functions. The Fr\'echet derivative in $u\in H$ against $h\in H$ is given by:
\[
-\frac{1}{N}\sum_{\xi\in E}\sum_{i\in \ZZ}c_\xi [\mu'(u(\frac iN))(\xi)(v_\xi-u(\frac{i}{N}))+ (1_{\xi}(r(i))-\mu(u(\frac iN))(\xi))]\phi(\frac{i}{N})h(\frac iN)
\]
which gives us the boundedness property uniformly in $r\in\RR$ and $u\in B$.
\end{proof}

\begin{proof}[Proof of Proposition \ref{Prop_Poi}]
For $u\in H$ fixed, $\Ba(u,\cdot)$ is an operator on $\R^\RR$ which is a space of finite dimension. The Fredholm alternative in such a space is $\text{Im}(\Ba)=(\text{ker}(\Ba^*))^\perp$. Moreover, by the ergodicity assumption, for any fixed $u\in H$, $\text{ker}(\Ba^*(u,\cdot))=\text{span}(\mu(u))$. Therefore equation (\ref{Poisson}) with $u\in H$ fixed has a solution if and only if:
\begin{equation*}
\int_\RR\mu(u)(dr)<G_{r}(u)-F(u),\phi>=0
\end{equation*}
This latter equality holds by the definition of $F$ and the bi-linearity of the duality product $<\cdot,\cdot>$. Moreover, we can then always choose $f$ satisfying $\int_\mathcal{R} f(u,r)\mu(u)(dr)=0$ by taking its projection on $(\text{ker}(\Ba^*))^\perp$. Thus we have a solution $f(u,\cdot)$ for any $u\in H$ fixed. Uniqueness of $f$ follows easily from the condition $\int_\mathcal{R} f(u,r)\mu(u)(dr)=0$.\\
 Recall that for all $\e\in]0,1]$, $u^\e\in B$ where $B=\{u\in H; \|u\|\leq C\}$ where $C$ is the deterministic constant given by Proposition \ref{Prop_bound}. Then the desired properties of $f$ follow from Lemma \ref{lem_Poi} and the fact that for a bounded domain $B\in H$, each function $u\in B\mapsto \alpha_{r(i),\zeta}(u(\frac iN))$ is bounded below and above  by strictly positive constants, Lipschitz continuous and Fr\'echet differentiable with Fr\'echet derivative uniformly bounded in $u$, $\zeta$ and $r$. 
\end{proof}

STEP 3 (Bound the martingale part). Our previous result implies that $f\in\mathcal{D}(\mathcal{A})$ and:
\begin{equation}\label{mart}
M^\e_t=f(u^\e_t,r^\e_t)-f(u_0,r_0)-\int_0^t\mathcal{A}f(u^\e_s,r^\e_s)ds
\end{equation}
defines a square-integrable martingale, see for example Ethier and Kurtz \cite{Eth_Kur}, chapter 4 Proposition 1.7.

\begin{proposition}\label{Prop_Var}
We have :
\begin{equation}\label{crochet}
<M^\e>_t=\frac{1}{\e}\int_0^t\sum_{i\in\ZZ}\sum_{\zeta\neq r_s(i)}[f(u^\e_s,r_s(r_s(i)\to\zeta))-f(u^\e_s,r^\e_s)]^2\alpha_{r_s(i),\zeta}(u^\e_s(\frac{i}{N}))ds
\end{equation}
\end{proposition}
\begin{proof}
It is a classical proof in stochastic calculus (see for example \cite{Eth_Kur}, Chapter 1, Problem 29). The It\^o and the Dynkin formulas give two distinct decompositions of the squared process $f^2(u^\e_t,r^\e_t)_{t\geq0}$ in a semi-martingale. The uniqueness of the Doob-Meyer decomposition of a semi-martingale enables us to identify the bracket of our martingale.
\end{proof}
We have the following semi-martingale decomposition:
\begin{equation}\label{dec_sm}
\int_0^t<G_{r^\e_s}(u^\e_s)-F(u^\e_s),\phi>ds=\e f(u^\e_t,r^\e_t)-\e f(u_0,r_0)-\e\int_0^t\frac{d}{ds}f(u^\e_s,r^\e_s)ds-\e M^\e_t
\end{equation}
Recall that we are interested in bounding $\EE\left| \int_0^t<G_{r^\e_s}(u^\e_s)-F(u^\e_s),\phi>ds\right|^2$. Since $\sup_{s\in[0,T]}|f(u^\e_s,r^\e_s)|$ is bounded independently of $\e\in]0,1]$, denoting by $\e g$ one of the two first terms  of (\ref{dec_sm}) we have $\EE|\e g|^2\leq C' \e^2$ where $C'$ is a constant independent of $\e\in]0,1]$. For the martingale term, by the It\^o isometry:
\[
\EE|M^\e_t|^2=\EE<M^\e>_t\leq\frac{1}{\e} C'
\]
since $\e<M^\e>_t$ is bounded uniformly in $t\in[0,T]$ and $\e\in]0,1]$ thanks to the bounds on $f$ and each $\alpha_{\xi,\zeta}$. More precisely we have:
\begin{eqnarray*}
&~&<M^\e>_t\\
&=&\frac{1}{\e}\int_0^t\sum_{i\in\ZZ}\sum_{\zeta\neq r_s(i)}[f(u^\e_s,r_s(r_s(i)\to\zeta))-f(u^\e_s,r^\e_s)]^2\alpha_{r_s(i),\zeta}(u^\e_s(\frac{i}{N}))ds\\
&\leq&\frac{4T}{\e}\alpha^+N|E|\sup_{s\in[0,T]}|f(u^\e_s,r^\e_s)|^2
\end{eqnarray*}

STEP 4 (Bound the finite variation part). It remains to bound the third term of our semi-martingale decomposition.
\begin{proposition}\label{Prop_Var}
There exists a constant $C$ independent of $\e\in]0,1]$ such that: 
\begin{equation*}
\int_0^T|\frac{d}{dt}f(u^\e_t,r^\e_t)|dt\leq C
\end{equation*}
almost surely.
\end{proposition}
\begin{proof}
Recall first in Theorem \ref{th_BR} the meaning of the notation $\frac{d}{dt}f(u^\e_t,r^\e_t)$. Recall that the map $u\mapsto f(u,r)$ for $r\in\RR$ fixed, is Fr\'echet differentiable with Fr\'echet derivative in $u$ denoted by $\frac{df}{du}[u,r]$ which is a bounded linear form on $H$. By the Riesz representation theorem there exists $f_u(u,r)\in H$ such that:
\[
\frac{df}{du}[u,r](h)=(f_u(u,r),h)_H\quad\forall h\in H
\] 
Moreover the correspondence is isometric: $\left\|\frac{df}{du}[u,r]\right\|_{H^*}=\left\|f_u(u,r)\right\|_H$. We then have (see Theorem 4 of \cite{Buck_Ried}):
\[
\frac{d}{dt}f(u^\e_t,r^\e_t)=<\Delta u^\e_t+G_{r^\e_t}(u^\e_t),f_u(u^\e_t,r^\e_t)>
\]
By Proposition \ref{Prop_Poi}, there exists a constant $C_1$ independent of $\e\in]0,1]$ and $t\in[0,T]$ such that:
\[
\left\|f_u(u^\e_t,r^\e_t)\right\|_H=\left\|\frac{df}{du}[u^\e_t,r^\e_t]\right\|_{H^*}\leq C_1
\]
Therefore:
\[
|\frac{d}{dt}f(u^\e_t,r^\e_t)|\leq C_1\|\Delta u^\e_t+G_{r^\e_t}(u^\e_t)\|_{H^*}\\
\]
It remains to show that $\|\Delta u^\e_t+G_{r^\e_t}(u^\e_t)\|_{H^*}$ is bounded uniformly in $\e$ and $t\in[0,T]$. For $\phi\in H$ we have, denoting by $D$ the derivative with respect to $x$:
\begin{eqnarray*}
&~&|<\Delta u^\e_t+G_{r^\e_t}(u^\e_t),\phi>|\\
&=&\left|-<D u^\e_t,D\phi>+\frac{1}{N}\sum_{\xi\in E}\sum_{i\in \ZZ}c_\xi 1_{\xi}(r(i))(v_\xi-u^\e_t(\frac{i}{N}))\phi\left({\frac{i}{N}}\right)\right|\\
&\leq&|(D u^\e_t,D\phi)_{L^2(I)}|+|E|\max_{\xi\in E}c_{\xi}(\max_{\xi\in E}|v_\xi|+C)C_P\|\phi\|_H\\
&\leq&\|Du^\e_t\|_{L^2(I)}\|D\phi\|_{L^2(I)}+|E|\max_{\xi\in E}c_{\xi}(\max_{\xi\in E}|v_\xi|+C)C_P\|\phi\|_H\\
&\leq&\|u^\e_t\|_H\|\phi\|_H+|E|\max_{\xi\in E}c_{\xi}(\max_{\xi\in E}|v_\xi|+C)C_P\|\phi\|_H\\
&\leq&(C+|E|\max_{\xi\in E}c_{\xi}(\max_{\xi\in E}|v_\xi|+C)C_P)\|\phi\|_H
\end{eqnarray*}
where $C$ is the deterministic constant given by Proposition \ref{Prop_bound}. Thus:
\[
\|\Delta u^\e_t+G_{r^\e_t}(u^\e_t)\|_{H^*}\leq C+|E|\max_{\xi\in E}c_{\xi}(\max_{\xi\in E}|v_\xi|+C)C_P
\]
This ends the proof.
\end{proof}

STEP 5 (All bounds together). Assembling all the bounds of the different terms we see that:
\[
\EE\left| \int_0^t<G_{r^\e_s}(u^\e_s)-F(u^\e_s),\phi>ds\right|^2\leq C\e
\]
with the constant $C$ independent of $\e\in]0,1]$. It remains to let $\e$ go to $0$ to conclude the proof of the identification of the limit.

STEP 6 (Uniqueness). We consider $u^1$ and $u^2$ two possible accumulation points verifying, for $i=1,2$:
\[
u^i_t=e^{\D t}u_0+\int_0^t e^{\D(t-s)}F(u^i_s)ds
\]
for all $t\in[0,T]$ almost surely. We can show easily that for any $t\in[0,T]$:
\[
\EE(\sup_{s\in[0,t]}\|u^1_s-u^2_s\|_H)\leq C\int_0^t \EE(\sup_{l\in[0,s]}\|u^1_l-u^2_l\|_H)dl
\]
with the constant $C$ independent of $t\in[0,T]$. Therefore,  by application of the Gronwall's lemma we obtain that $\EE(\sup_{t\in[0,T]}\|u^1_t-u^2_t\|_H)=0$.
\end{proof}

\subsection{The general case}\label{Cas_g}

We now consider the original case with different state classes $E_1,\cdots, E_l$. The first thing to do is to prove the tightness of the family $((u^\e,\bar{r}^\e),\e\in]0,1])$ in $\mathbb{D}([0,T],H\times\{1,\cdots,l\}^{|\ZZ|})$ (see section \ref{MAAE} for the definition of $\bar{r}^\e$). We apply a comparison argument. We notice that:
\[
\PP(r^\e_{t+h}(i)=\zeta|r^\e_t(i)=\xi)=\left\{\begin{array}{cc}
\frac1{\e}\alpha_{\xi,\zeta}\left(u_t\left(\frac{i}{N}\right)\right)h+o(h)&\text{ if $\xi,\zeta$ are in the same class $E_k$}\\
\alpha_{\xi,\zeta}\left(u_t\left(\frac{i}{N}\right)\right)h+o(h)&\text{ otherwise}
\end{array}
\right.
\]
For $\xi,\zeta\in E$ we set:
\[
\lambda_{\xi,\zeta}=\left\{\begin{array}{cc}
\frac1{\e}\alpha^+&\text{ if $\xi,\zeta$ are in the same class $E_k$}\\
\alpha^+&\text{ otherwise}
\end{array}
\right.
\]
We denote by $r^{\e,\text{Max}}$ the associated jump process with constant rates $\lambda_{\xi,\zeta}$. We have the following stochastic domination:
\[
\PP(r^\e_{t+h}(i)=\zeta|r^\e_t(i)=\xi)\leq\PP(r^{\e,\text{Max}}_{t+h}(i)=\zeta|r^{\e,\text{Max}}_t(i)=\xi)
\]
We construct the process $\bar{r}^{\e,\text{Max}}$ as $\bar{r}^\e$ by aggregation. The sequence $(\bar{r}^{\e,\text{Max}},\e\in]0,1])$ is tight in $\mathbb{D}([0,T],\{1,\cdots,l\}^{|\ZZ|})$, see for instance \cite{Yin_Zhang}. Therefore, by comparison, the sequence $(\bar{r}^{\e},\e\in]0,1])$ is also tight in $\mathbb{D}([0,T],\{1,\cdots,l\}^{|\ZZ|})$, see for instance \cite{Ja_Sh}. Moreover, the sequence $(u^\e,\e\in]0,1])$ is tight in $\mathbb{D}([0,T],H)$ by applying the arguments developed for the "all-fast" case. Endowing $H\times\{1,\cdots,l\}^{|\ZZ|}$ with the product topology we see that the sequence $((u^\e,\bar{r}^\e),\e\in]0,1])$ is tight in $\mathbb{D}([0,T],H\times\{1,\cdots,l\}^{|\ZZ|})$. 

We must now deal with the identification of the limit. There are six steps in the proof of the identification of the limit, we have to check that these six steps generalize to the general case (jumping slow-fast case).

In STEP 1, by the same arguments, we obtain that it is sufficient to show that:
\[
\lim_{\e\to0}\EE\left| \int_0^t<G_{r^\e_s}(u^\e_s)-F_{\bar{r}^\e_s}(u^\e_s),\phi>ds\right|^2=0
\]

In STEP 2, the Poisson equation becomes: 
\begin{eqnarray}\label{Poisson_g}
\mathcal{B}f(u,r)&=&<G_{r}(u)-F_{\bar r}(u),\phi>
\end{eqnarray}
where $\Ba$ is now the "fast" part of the generator ($f$ does not depend on $\bar r$ explicitly because $\bar r$ is constructed from $r$). A configuration $r$ for the ion channels is now:
\[
r\in E_{j_1}\times\cdots\times E_{j_{N-1}}
\]
where $(j_1,\cdots,j_{N-1})\in \{1,\cdots,l\}^{N-1}$ (noting that $|\ZZ|=N-1$). That is each channel is in one of the class $E_j$ for $j\in\{1,\cdots,l\}$. Then, since for fixed $u$ the quasi-stationary measure associated to the class $E_{j_k}$ for $k\in\{1,\cdots,N-1\}$ is $\mu_{j_k}(u)$, we have that the kernel of $\mathcal{B}^*$ is spanned by:
\[
\{\mu_{j_1}(u)\times\cdots\times\mu_{j_{N-1}}(u),(j_1,\cdots,j_{N-1})\in \{1,\cdots,l\}^{N-1}\}
\]
 when $u$ is held fixed. Then, by the Fredholm alternative and the definition of the averaged function $F_{\bar r}(u)$, the Poisson equation (\ref{Poisson_g}) has a solution. Uniqueness follows by the projection condition:
\begin{equation}
\int_{\RR}f(u,r)\mu_{j_1}(u)(dr)\times\cdots\times\mu_{j_{N-1}}(u)(dr)=0,\quad \forall (j_1,\cdots,j_{N-1})\in \{1,\cdots,l\}^{N-1}
\end{equation}

The proof of the last three steps is then exactly the same as in the "all fast" case.

\section{Example}\label{section_ex}

In this section we give a concrete example where our result allows to reduce the complexity of a neuronal model staying nevertheless at a stochastic level.

We consider a usual Hodgkin-Huxley model but, to be very straightforward in the application of our result, we only consider the sodium current i.e. all the ion channels are sodium channels. It is still a case of interest since sodium channels are involved in the increasing phase of an action potential. At a fixed potential, the kinetic of a sodium channel is described by Figure \ref{kin_Na} which represents the states and the jump rates of a continuous time Markov chain denoted by $r^\e(i)$ for the channel at position $\frac{i}{N}$ for $i\in\ZZ$.  A sodium channel can be in $8$ different states denoted by $m_ih_j$ for $i\in\{0,1,2,3\}$ and $j\in\{0,1\}$, the state $m_3h_1$ is the only "open" state for the sodium channel (see \cite{Hille} for more details). We simply write $a_m,b_m$ and $a_h,b_h$ for $a_m(u),b_m(u)$ and $a_h(u),b_h(u)$ (see Appendix \ref{app} for more details on the rate functions). $m$ is the fast variable and $h$ the slow variable for the kinetic of a sodium channel. 
\begin{figure}
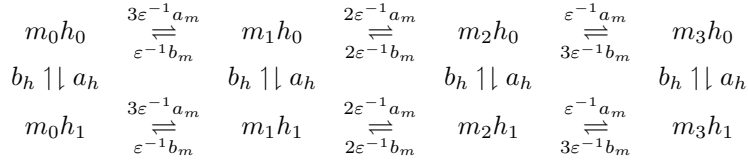

\[
\begin{array}{ccccccc}
m_0h_0&\underset{\e^{-1}b_m}{\overset{3\e^{-1}a_m}{\rightleftharpoons}}&m_1h_0&\underset{2\e^{-1}b_m}{\overset{2\e^{-1}a_m}{\rightleftharpoons}}&m_2h_0&\underset{3\e^{-1}b_m}{\overset{\e^{-1}a_m}{\rightleftharpoons}}&m_3h_0\\
b_h\upharpoonleft\downharpoonright a_h&~&b_h\upharpoonleft\downharpoonright a_h&~&b_h\upharpoonleft\downharpoonright a_h&~&b_h\upharpoonleft\downharpoonright a_h\\
m_0h_1&\underset{\e^{-1}b_m}{\overset{3\e^{-1}a_m}{\rightleftharpoons}}&m_1h_1&\underset{2\e^{-1}b_m}{\overset{2\e^{-1}a_m}{\rightleftharpoons}}&m_2h_1&\underset{3\e^{-1}b_m}{\overset{\e^{-1}a_m}{\rightleftharpoons}}&m_3h_1
\end{array}
\]
\caption{Kinetic of sodium channels.}\label{kin_Na}
\end{figure}

Each channel can therefore be in one of these eight states divided in two classes : $E=E_0\sqcup E_1$ where $E_0=\{m_0h_0,m_1h_0,m_2h_0,m_3h_0\}$ and $E_1=\{m_0h_1,m_1h_1,m_2h_1,m_3h_1\}$. Inside these two classes the states communicate at fast rates and transition between these two classes occurs at slow rates. A sodium channel is open if and only if it is in the state $m_3h_1$. We then define $c_{m_3h_1}=c_{\text{Na}}$,  $v_{m_3h_1}=v_{\text{Na}}$ and $c_\xi=0$, $v_\xi=0$  when $\xi\in E\setminus\{m_3h_1\}$.

The equation (\ref{dyn}) describing the evolution of the potential, for a number $N$ of ion channels, is here given by :
\begin{equation}\label{ex}
\partial_t u^\e=K\D u^\e + \frac{1}{N}\sum_{i\in \ZZ}c_{\text{Na}} 1_{m_3h_1}(r^\e(i))(v_{\text{Na}}-u(\frac{i}{N}))\delta_{\frac{i}{N}}
\end{equation}
The global variable $u$ represents the difference of potential between the inside and the outside of the axon membrane.  We see that Equation (\ref{ex}) is indeed of the form that we have studied in this paper. $K$ is a constant related to the radius and the internal conductivity of the axon (see Appendix \ref{app}).

Let us compute the different generators and quasi-invariant measures associated to this model according to our description. The two "fast" generators are the same and are given for $u(\frac iN)$ fixed (again, we do not make the dependence appear), by :
\[
\Ba_j=\left(\begin{array}{cccc}
-3a_m&3a_m&0&0\\
b_m&-b_m-2a_m&2a_m&0\\
0&2b_m&-2b_m-a_m&a_m\\
0&0&3b_m&-3b_m
\end{array}
\right)
\]
with $j=0$ or $1$. We can compute the associated quasi-invariant measure $\mu_j(u(\frac iN))$ ($j=0$ or $1$), the only term of interest for us is :
\[
\mu_1(u(\frac iN))(m_3h_1)=\frac{1}{\left(1+\frac{b_m(u(\frac iN))}{a_m(u(\frac iN))}\right)^3}
\]
The reader familiar with the classical Hodgkin-Huxley model will notice that $\mu_1(u(\frac iN))(m_3h_1)$ is in fact equal to the steady-state function associated to the ODE describing the motion of the $m$-gates in the classical deterministic Hodgkin-Huxley description, see for example \cite{HH}. 

The asymptotic aggregated Markov chain $\bar r$ is valued in the two-state space $\{0,1\}$. According to Proposition \ref{YZ} its generator is:
\[
\left(\begin{array}{cc}
-a_h(u(\frac iN))&a_h(u(\frac iN))\\
b_h(u(\frac iN))&-b_h(u(\frac iN))
\end{array}
\right)
\]
Therefore, according to Theorem \ref{APDE}, the reduced model is described by the following PDE coupled with the continuous-time Markov chain $\bar r$ :
\begin{equation}\label{ex_av}
\partial_t u=K\D u + \frac{1}{N}\sum_{i\in \ZZ}1_{1}(\bar r(i))c_{\text{Na}} \mu_1(u(\frac{i}{N}))(m_3h_1)(v_{\text{Na}}-u(\frac{i}{N}))\delta_{\frac{i}{N}}
\end{equation}
We display a realization of the averaged piecewise deterministic Markov process in Figure \ref{fig2}. To perform the numerical simulation (in C) we extended Riedler's work (cf. \cite{Ried}) to our particular framework. In  \cite{Ried}, an algorithm (Algorithm A2) is proposed to simulate the trajectory of a PDMP in finite dimension. The idea is to simulate the jumping part of the PDMP and to use an accurate method to simulate the ODE constituting the deterministic part of the PDMP. The kinetic of the jumping component for PDMP in finite and infinite dimensions have the same form. Here we simulate the jumping part of our infinite dimensional PDMP following \cite{Ried} and simulate the PDE between successive jumps by a deterministic scheme.  For the PDE, we used  an explicit finite difference Euler scheme in space and time. Simulating the jumps of an inhomogeneous time Markov chain is classical and there exist a lot of efficient numerical schemes for PDE. Therefore the described method is natural. This argument is of course not a proof but gives an heuristic interpretation of our approach.
   
\begin{figure}
\begin{center}
\includegraphics[width=10cm]{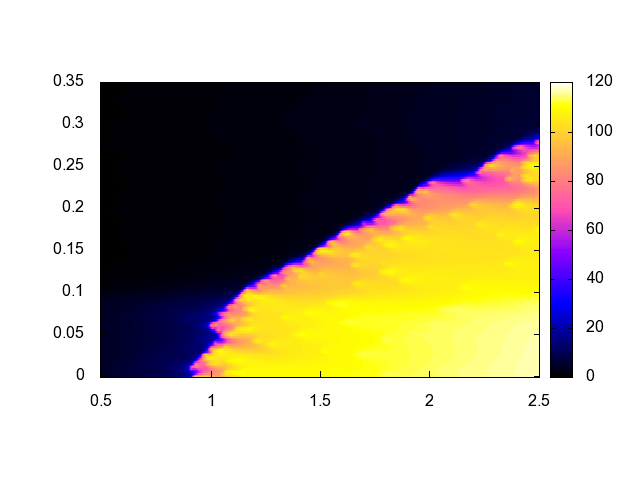}
\caption{Simulation of the action potential against space (vertical axis) and time (horizontal axis).  The averaged model (\ref{ex_av}) is displayed for a number of ion channels fixed to $250$. We see what we expect from the partial differential equation: a traveling wave connecting the initial state $0$ to the stable state $v_\text{Na}$.}\label{fig2}
\end{center}
\end{figure}

\section{Concluding remarks}

In this paper, we apply the methods of stochastic averaging to a fully coupled PDMP in infinite dimensions. To the best of our knowledge it is the first time that such a problem is considered. The two main difficulties are due, on one hand to the infinite dimensional setting and on the other hand to the fact that the system is fully-coupled. Regarding the fully-coupled property, our task was made easier by the uniform bound on the macroscopic component of the PDMP and on the jump rates of the continuous time Markov chain. As an application,  we have shown how to reduce the complexity of a generalized Hodgkin-Huxley model which can be used as a paradigm in the modelling of the propagation of an action potential for neurons but also for cardiac cells.

Let us make some comments about a possible extension of our work to other forms of generalized Hodgkin-Huxley models. Firstly, to choose that the ion channels are uniformly distributed along the axon ($\frac iN$ for $i\in\ZZ$) is just a simple way to fix ideas. We could also choose any distribution in a finite subset $\mathcal{N}$ of $\mathring{I}$ without  changing anything to our work except the notations.

Secondly, in \cite{Buck_Ried}, the authors noticed that it can be of interest to replace the Dirac point masses $\delta_
{\frac iN}$ of the spatial stochastic Hodgkin-Huxley model by functions of the form $|U(\frac iN)|^{-1}1_{U(\frac iN)}(\cdot)$ where $U$ is a neighborhood of the point $\frac{i}{N}$. Indeed, the concentration of ions is essentially homogeneous in a spatially extended domain around an open channel. Along the same line, one can argue that we could replace the Dirac point 
masses $\delta_{\frac iN}$ by any smoother approximation $\phi_{\frac iN}$. The mathematics become then a little easier. The methods presented in this paper apply to both cases.

The next step after the present averaging result is to prove the associated central limit theorem. That is to say give more precision on how the original model converges in law to the averaged model.

Another line of work can be to consider that not only some ion channels but also the membrane potential have a faster kinetic. It seems that this situation is closer to the real behavior of the nerve impulse \cite{Hille}. This last question is related to the study of invariant measures for PDMP and the necessity to find explicit form or bounds for these invariant measures. These will be the object of a future work.
\\
\\
{\bf Acknowledgments :} the authors warmly thank the anonymous referee for his/her helpful comments,  suggestions and corrections which led to a significant improvement of the previous versions of the manuscript.

\appendix

\section{Austin's lemma about the semi-group $e^{\D t}$}\label{app_lem}

We recall here a lemma in \cite{Austin}. Note that the last part of conclusion 1. of the following lemma is contained in the proof of the corresponding lemma in \cite{Austin}.

\begin{lemma}\label{Aulem}
Let $y$ be in the interior of $I$.  Then $e^{\Delta t}\delta_y$ is a smooth function on $I$ vanishing at the end-points for any $t > 0$. Furthermore:
\begin{enumerate}
\item there is some constant $C_1 > 0$, depending on $T$ but otherwise not on $t \in [0,T]$, such that for any continuous function $u:[0,t] \to\R$ we have that the function
\[I \to \R: x \mapsto \int_0^t u(s) e^{\Delta (t-s)}\delta_y(x) ds\]
is in $H$ and satisfies the estimate
\[\bigg\|\int_0^t u(s)  e^{\Delta (t-s)}\delta_y(\cdot)ds\bigg\|_H \leq C_1\|u\|_\infty\]
Moreover for any $\eta>0$ we can choose $\e>0$ so small that:
\[
\bigg\|\int_{t-\e}^t u(s)  e^{\Delta (t-s)}\delta_y(\cdot)ds\bigg\|_H \leq \frac12\eta\|u\|_\infty
\]
\item for any fixed $\varepsilon > 0$ there is some constant
$C_2(\varepsilon)$, depending on $\varepsilon$ and $T$ but
otherwise not on $t \in [0,T]$, such that for any continuous
function $u:[0,t] \to \R$ we have
\[\bigg\|\int_0^{t-\varepsilon} u(s)  e^{\Delta (t-s)}\delta_y(\cdot) ds\bigg\|_H \leq C_2(\varepsilon)\int_0^{t-\varepsilon}|u(s)|ds\]
for any $t \in [0,T]$.
\end{enumerate}
\end{lemma}

\section{Two theorems about tightness in Hilbert spaces}\label{Met}

We recall here two important results of \cite{metivier} about tightness in Hilbert spaces with notations and choice of spaces adapted to our need.

\begin{theorem}[General criterion for tightness]\label{GCT}
Let us assume that $(u^\e,\e\in]0,1])$ satisfies Aldous's condition which means that for any $\delta,M>0$, there exist $\eta,\e_0>0$ such that for all stopping times $\tau$ such that $\tau+\eta<T$:
\begin{equation}\label{Ald}
\sup_{\e\in]0,\e_0]}\sup_{\theta\in]0,\eta[}\PP(\|u^\e_{\tau+\theta}-u^\e_\tau\|_H\geq M)\leq\delta
\end{equation}
If moreover, for each $t\in[0,T]$ fixed the family $(u^\e_t,\e\in]0,1])$ is tight in $H$ then $(u^\e,\e\in]0,1])$ is tight in $\mathbb{D}([0,T],H)$.
\end{theorem}
\begin{theorem}[Tightness in a Hilbert space]\label{TH}
$H$ is a separable Hilbert space endowed with a basis $\{e_k,k\geq1\}$. We denote by $L_k$ for $k\geq1$:
\[
L_k=span\{e_i,1\leq i\leq k\}
\]
Then  $(u^\e_t,\e\in]0,1])$ is tight in $H$ if, and only if, for any $\delta,\eta>0$ there is $\rho,\e_0>0$ and $L_{\delta,\eta}\subset \{L_k,k\geq1\}$ such that :
\begin{enumerate}
\item\begin{equation}\label{TH1} \sup_{\e\in]0,\e_0]} \PP(\|u^\e_t\|_H>\rho)\leq\delta\end{equation}
\item \begin{equation}\label{TH2} \sup_{\e\in]0,\e_0]} \PP(d(u^\e_t,L_{\delta,\eta})>\eta)\leq\delta\end{equation}
\end{enumerate}
where $d(u^\e_t,L_{\delta,\eta})=\inf_{v\in,L_{\delta,\eta}}\|u^\e_t-v\|_H$.
\end{theorem}

\section{Functions and data for the simulation}\label{app}

For the simulation, the jump rate functions are given, for real $u$, by :
\begin{eqnarray*}
a_m(u)&=&\frac{0.1(25-u)}{e^{2.5-0.1u}-1}\\
b_m(u)&=&4e^{\frac{-u}{18}}\\
a_h(u)&=&0.07e^{\frac{-u}{20}}\\
b_h(u)&=&\frac{1}{e^{3-0.1u}+1}
\end{eqnarray*}
The maximal conductance associated to sodium ions is $c_{\text{Na}}=120 \text{ mS.cm}^{-2}$ and the potential at rest is $v_{\text{Na}}=115$ mV. The constant $K$ is given by $K=\frac{a}{2R}$ where $a$ is the radius of the axon, $a=0.0238$ cm and $R$ is the internal resistance of the axon, $R=34.5$ $\Omega$.cm, these data are classical, see for example \cite{HH}. We have used an input on the potential equal to $\mu=6.7$ all along the time on the segment $[0,0.1]$ of the axon.

\end{document}